\documentclass[11pt]{article}
\usepackage{amsfonts,amsmath,amsthm,amssymb,enumerate,color,mathrsfs}
\usepackage{tocstyle}

\usepackage[makeroom]{cancel}
\usepackage{authblk}

mm \evensidemargin 2.5 true mm

\numberwithin{equation}{section} \newcommand{\beq}{\begin{equation}}
  \newcommand{\eeq}{\end{equation}}
\newcommand{\bea}{\begin{eqnarray}} \newcommand{\eea}{\end{eqnarray}}
\newcommand{\beas}{\begin{eqnarray*}}
  \newcommand{\eeas}{\end{eqnarray*}} \newcommand{\ds}{\displaystyle}
\newtheorem{theorem}{Theorem}[section]

\newtheorem{definition}[theorem]{Definition}
\newtheorem{proposition}[theorem]{Proposition}

\newtheorem{corollary}[theorem]{Corollary}
\newtheorem{lemma}[theorem]{Lemma}
\newtheorem{remark}[theorem]{Remark}
\newtheorem{example}[theorem]{Example}
\newtheorem{examples}[theorem]{Examples}
\newtheorem{foo}[theorem]{Remarks}

\newcommand{\ang}[1]{\left<#1\right>} 

\newcommand{\brak}[1]{\left(#1\right)} 
\newcommand{\crl}[1]{\left\{#1\right\}} 
\newcommand{\edg}[1]{\left[#1\right]} 

\newcommand{\E}{\mathbb{E}} \renewcommand\P{\mathbb{P}}

 \newcommand{\bM}{\mathbb M}
 
 \newcommand{\F}{\mathcal F}
\newcommand{\Ho}{\mathcal H} \newcommand{\V}{\mathcal V}
 \newcommand{\M}{\mathbb M}
 
 \newcommand{\R}{\mathbb R}
 \newcommand{\be}{\beta}
 \newcommand{\ve}{\varepsilon}

\newcommand{\ch}{\mathcal H}

\newcommand{\ep}{\varepsilon} 

\DeclareMathOperator{\Cut}{Cut}

\title{Radial processes for sub-Riemannian Brownian motions and
  applications}

\author[1]{Fabrice Baudoin} \author[2]{Erlend Grong}
\author[3]{Kazumasa Kuwada\footnote{Deceased on December 29, 2018.}}
\author[4]{Robert Neel} \author[5]{Anton Thalmaier}
\affil[1]{\small Department of Mathematics, University of
  Connecticut,\par
  341 Mansfield Road, Storrs, CT 06269-1009, USA\par
  \texttt{fabrice.baudoin@uconn.edu}\vspace{1em}}

\affil[2]{\small University of Bergen, Department of Mathematics, \par
  P.O. Box 7803, 5020 Bergen, Norway\par
  \texttt{erlend.grong@math.uib.no}\vspace{1em}}

\affil[3]{\small Department of Mathematics, Graduate School of
  Science,\par
  Tohoku University, 980-8578, Sendai, Japan\vspace{1em}}
  
\affil[4]{\small Department of Mathematics, Lehigh University,\par
  Bethlehem, PA 18015-3174, USA. \par
  \texttt{robert.neel@lehigh.edu}\vspace{1em}}

\affil[5]{\small Department of Mathematics, University of
  Luxembourg,\par
  L-4364 Esch-sur-Alzette, Luxembourg\par
  \texttt{anton.thalmaier@uni.lu}}

\date{\vspace{-8,5ex}}

\begin{document}

{ \makeatletter
  \addtocounter{footnote}{1} 
  \renewcommand\thefootnote{\@fnsymbol\c@footnote}%
  \makeatother
  \maketitle
}

\newcommand\blfootnote[1]{%
  \begingroup \renewcommand\thefootnote{}\footnote{#1}%
  \addtocounter{footnote}{-1}%
  \endgroup } \blfootnote{The first author has been supported in part
  by NSF Grants DMS-1660031, DMS DMS-1901315 and the Simons
  Foundation; the second author by project 249980/F20 of the
  Norwegian Research Council; the third author by JSPS
  Grant-in-Aid for Young Scientist (KAKENHI) 26707004; the fourth
  author by a grant from the Simons Foundation (\#524713,
  RN), and the fifth author by FNR Luxembourg: OPEN project
  GEOMREV O14/7628746.}

\begin{abstract}
  We study the radial part of sub-Riemannian Brownian motion in the
  context of totally geodesic foliations. It\^o's formula is proved
  for the radial processes associated to Riemannian distances
  approximating the Riemannian one. We deduce very general stochastic
  completeness criteria for the sub-Riemannian Brownian motion. In the
  context of Sasakian foliations and H-type groups, one can push the analysis further,
  and taking advantage of the recently proved sub-Laplacian comparison
  theorems one can compare the radial processes for the sub-Riemannian
  distance to one-dimensional model diffusions. As a geometric
  application, we prove Cheng's type estimates for the Dirichlet
  eigenvalues of the sub-Riemannian metric balls, a result which seems
  to be new even in the Heisenberg group.\\
\end{abstract}

\section{Introduction}

In the context of Riemannian manifolds the study of the radial part
of Brownian motion yields new proofs and sheds new light on
several well-known theorems of Riemannian geometry; see for instance the
paper \cite{Ich} and the book \cite{Hsu} for an overview. Our goal in
the present paper is to extend those techniques to the context of
sub-Riemannian manifolds. In the last few years, the study of Brownian motion on sub-Riemannian
manifolds has attracted a lot of interest, see \cite{Bau}, \cite{BFG},
\cite{GL} and \cite{Th}, and several applications to the study of heat
semigroup gradient bounds and functional inequalities on pathspaces
have been obtained. Despite those numerous works the probabilistic
study of the radial part of the sub-Riemannian Brownian motion is not yet
developed. Taking advantage of the sub-Laplacian comparison
theorems recently proved in \cite{BGKT}, it is now possible to pursue
such a study.

\

In this paper we focus on two classes of sub-Riemannian manifolds. The
first class is the class of sub-Riemannian manifolds whose horizontal
distribution is the horizontal distribution of some Riemannian
foliation with totally geodesic leaves. The sub-Riemannian geometry of
such structures is by now well understood, thanks to the works
\cite{BG}, \cite{GT1} and \cite{GT2}. A key insight is to approximate
the sub-Riemannian distance $d_0$ by a family of Riemannian distances
$d_\ve$, $\ve >0$ which converges to $d_0$ as $\ve \to 0$. The
sub-Laplacian comparison theorems associated to $d_\ve$ obtained in
that context are very general but with the drawback that there is no
limit when $\ve \to 0$. Hence, no results for the sub-Laplacian of the
sub-Riemannian distance can be deduced.  One of the main results we
obtain for the radial processes is Theorem \ref{th:Ito-radial}, which
is the It\^o formula for the radial process. It reads as
follows. Let $(\xi_t)_{t \ge 0}$ be the sub-Riemannian Brownian motion
on a Riemannian manifold equipped with a totally geodesic
foliation, and let $\Delta_{\ch}$ be the sub-Laplacian (see the next section for the precise definitions). Denote by $r_\ve$ the distance from a fixed point $x_0$ and
by $\zeta$ the lifetime of the process. Then
\begin{equation*}
  r_\ve ( \xi_{ t \wedge \zeta } ) 
  = 
  r_\ve ( \xi_0 ) 
  + \be_t 
  + \int_0^{ t \wedge \zeta } \Delta_{\ch} r_\ve ( \xi_s ) ds
  - l_{ t \wedge \zeta }
\end{equation*}
where $l_t$ is a non-decreasing continuous process which increases
only when $\xi_t$ is in the $d_\ve$ cut-locus of $x_0$ and where
$\be_t$ is a martingale on $\R$, starting from $0$, with $\ang{\be}_t \le 2 t$. This
decomposition is the sub-Riemannian analogue of
Kendall's well-known result \cite{Kendall}. However, note that $\be_t$ in this
result is not a Brownian motion unlike in the Riemannian case. Even if
$\be_t$ is not a Brownian motion, we are still able in Section~\ref{sec:ComparisonGen} to prove very general stochastic completeness criteria, see
Theorem \ref{thm:S-cpl}.

\
 
The second class of sub-Riemannian manifolds we will focus on is the
class of Sasakian foliations. Sasakian foliations are a special class
of totally geodesic foliations for which the leaves have dimension
one. In that particular class of examples, it was proved in
\cite{BGKT,LL} that it is possible to prove sharp sub-Laplacian
comparison theorems for $\Delta_{\ch} r_\ve$ which actually have a
limit when $\ve \to 0$. As a consequence, we are able to study the
radial process with respect to the sub-Riemannian distance itself. Let $n$ be the dimension of the horizontal distribution.
Our main result is the comparison Theorem~\ref{comparison 1d}. It
states that under natural curvature lower bounds (expressed in terms of constants $k_1$ and $k_2$), one has in a weak
sense,
\[
  r_0 ( \xi_{ t } ) \le \tilde{\xi}_t
\]
where $\tilde{\xi}$ is a one-dimensional diffusion with generator
\[
  L_{k_1,k_2}= \big(F_{\mathrm{Sas}}(r,k_1) + (n-2)
  F_{\mathrm{Rie}}(r,k_2)\big) \frac{\partial}{\partial r}
  +\frac{\partial^2}{\partial r^2}
\]
and $F_{\mathrm{Sas}}, F_{\mathrm{Rie}}$ are the explicit functions
defined by
\begin{equation} \label{FRie} F_{\mathrm{Rie}}(r,k) = \begin{cases}  \sqrt{k} \cot \sqrt{k} r & \text{if $k > 0$,} \\
  {1}/{r} & \text{if $k = 0$,}\\ \sqrt{|k|} \coth \sqrt{|k|} r &
  \text{if $k < 0$,} \end{cases}\end{equation} and
\begin{equation} \label{FSas}
F_{\mathrm{Sas}}(r,k) = \begin{cases}  \displaystyle\frac{\sqrt{k}(\sin \sqrt{k}r  -\sqrt{k} r \cos \sqrt{k} r)}{2 - 2\cos \sqrt{k} r - \sqrt{k} r \sin \sqrt{k} r} & \text{if $k > 0$,} \\
  {4}/{r} & \text{if $k = 0$,}\\ \displaystyle\frac{\sqrt{|k|}(
    \sqrt{|k|} r \cosh \sqrt{|k|} r - \sinh \sqrt{|k|}r)}{2 - 2\cosh
    \sqrt{|k|} r + \sqrt{|k|} r \sinh \sqrt{|k|} r} & \text{if
    $k < 0$.} \end{cases}\end{equation}

For instance, in the case of the 3-dimensional Heisenberg group, which
is a Sasakian manifold for which $n=2$ and $k_1=0, k_2=0$ one can see
that
\[
  L_{0,0}=\frac{4}{r} \frac{\partial}{\partial r}
  +\frac{\partial^2}{\partial r^2}.
\]
As a consequence, the sub-Riemannian radial part of the
sub-Riemannian Brownian motion in the Heisenberg group can be
controlled by a 5-dimensional Bessel processes. Note that the dimension
5 here is not too surprising since 5 is the MCP dimension of the
Heisenberg group (see \cite[Section~3.6]{BGKT} for further comments about the MCP
dimension in that context). As a corollary of our comparison theorem, we obtain a Cheng's type
estimate for the Dirichlet eigenvalues of sub-Riemannian metric balls,
see Section \ref{sec Cheng}. In the case $k_1=k_2=0$, which
thus includes the Heisenberg group, our result becomes the following.

\begin{theorem}
  Let $\M$ be a sub-Riemannian manifold associated to a Sasakian foliation, with horizontal distribution of dimension $n$, and which satisfies the above comparison result with $k_1=k_2=0$.
  For $x_0 \in \M$ and $R>0$, let $\lambda_1( B_0(x_0,R))$ denote the
  first Dirichlet eigenvalue of the sub-Riemannian ball $B_0(x_0,R)$
  and let $\tilde{\lambda}_1(d,R) $ denote the first Dirichlet
  eigenvalue of Euclidean ball with radius $R$ in $\mathbb{R}^d$.
  Then, for every $x_0 \in \M$ and $R>0$
  \[
  0<  \lambda_1( B_0(x_0,R)) \le \tilde{\lambda}_1( n+3,R).
  \]
\end{theorem}

The structure of the paper is as follows. In Sections~\ref{sec:Preliminaries}, \ref{sec:RadialProcess} and \ref{sec:ComparisonGen} we consider the general case of totally geodesic Riemannian foliations. After introducing the necessary background on such manifolds in Section~\ref{sec:Preliminaries}, we describe the diffusion and drift part of the sub-Riemannian radial process in Section~\ref{sec:RadialProcess}. Such a representation allows us to give a criterion for non-explosion of the sub-Riemannian Brownian motion in Section~\ref{sec:ComparisonGen}, which is more general than previous criteria for stochastic completeness found in \cite{BG} and \cite{GT3}. Finally, we use the sharp comparison theorem available to us for the case of Sasakian manifolds and H-type groups to prove results on the first Dirichlet eigenvalues and on exit times of sub-Riemannian balls in Section~\ref{sec:Sasakian}.

\section{Preliminaries and assumptions} \label{sec:Preliminaries}

In this preliminary section we introduce the geometric framework and
recall some of the general sub-Laplacian comparison theorems obtained
in \cite{BGKT}.

\subsection{Totally geodesic Riemannian foliations and canonical
  variation}

Let $(\M,g)$ be a complete Riemannian manifold of dimension $n+m$
equipped with a foliation $\mathcal{F}$ with $m$-dimensional
leaves. Let $\V$ be the integrable subbundle tangent to the leaves of
$\mathcal{F}$ and write its orthogonal complement with respect to $g$
as $\Ho$. We will assume that the foliation is \emph{Riemannian} and
with \emph{totally geodesic leaves}, which is equivalent to the
assumptions that
$$(\mathcal{L}_X g)(Z,Z) = 0, \qquad (\mathcal{L}_Z g)(X,X) = 0, \qquad X \in \Gamma^\infty(\Ho),\ Z \in \Gamma^\infty(\V),$$
where $\mathcal L$ stands for the Lie derivative. For further details
about totally geodesic foliations we refer to \cite{B2}. We will also
consider \emph{the canonical variation} $g_\ve$ of the Riemannian
metric $g$ defined by
$$g_\ve = g_\Ho \oplus \frac{1}{\ve} g_\V, \qquad g_\Ho = g|\Ho,\ g_\V = g|\V, \qquad \ve > 0.$$ We let $d_{\ve}$ be the Riemannian distance associated to $g_{\ve}$. The limit $\ve \to 0$ is called \emph{the sub-Riemannian limit}.  Throughout this paper, we will assume that  $\Ho$ is \emph{bracket-generating}, i.e.\ we assume that elements in $\Gamma^\infty(\Ho)$ together with all possible iterated brackets of such vector fields span the entire tangent bundle~$T\M$. If this is the case, the limiting distance $d_0(x,y) = \lim_{\ve \downarrow 0} d_\ve(x,y)$ will always be finite, is called the sub-Riemannian distance and has the following alternative realization. An absolutely continuous curve $\gamma: [0,t_1] \to \M$ is called \emph{horizontal} if $\dot \gamma(t) \in \Ho_{\gamma(t)}$ for almost every $t \in [0,t_1]$. It is clear that the length on horizontal curves only depends on $g_\Ho$. The bracket-generating condition ensures that any pair of points can be connected by a horizontal curve and the distance $d_0(x,y)$ can be realized as the infimum of the lengths of all horizontal curves connecting the pair of points.

\

For any fixed $x \in \M$, define $r_{\ve}(y) = d_\ve(x,y)$ for any $\ve\ge 0$. We further assume that there are no non-trivial abnormal minimizers for the sub-Riemannian limit; note that this is known to always hold in the Sasakian case \cite[Chapt.~8]{ABB}. The
cut locus $\Cut_\ve (x)$ is defined such that
$ y \in \M \setminus \Cut_\ve (x)$ if there exists a unique, non-conjugate, length-minimizing geodesic from $x$ to $y$ relative to $g_\ve$. The global
cut locus of $\M$ is defined by
\[
  \Cut_\ve (\M)=\left\{ (x,y) \in \M \times \M,\ y \in \Cut_\ve (x)
  \right\}.
\]

So far, the geometry and topology of $\Cut_0 (\M)$ is only poorly
understood. However, the following is known:

\begin{lemma}[\cite{A}, \cite{RT}]\label{cutlocus}
  Let $\ve \ge 0$. The following statements hold:
  \begin{enumerate}[\rm (a)]
  \item The distance function $x \to d_\ve (x_0,x) $ is locally
    semi-concave in $\M \setminus \{ x_0 \}$. In particular, it is
    twice differentiable almost everywhere.
  \item The set $\M \setminus \Cut_\ve (x_0)$ is open and dense in
    $\M$.
  \item The function $(x,y) \to d_\ve (x,y)^2$ is smooth on
    $\M \times \M \setminus \Cut_\ve (\M)$.
  \end{enumerate}
\end{lemma}

The following theorem can be found in \cite{BGMR}.

\begin{theorem}[\cite{BGMR}]\label{p:cutapprox}
  Let $x,y \in \M$ with $y \notin \Cut_0(x)$. Then there exists an
  open neighbourhood $V$ of $y$ and $\varepsilon'>0$ such that
  $V \cap \Cut_\ve(x) = \emptyset$ for all $0\leq \ve <
  \ve'$. Furthermore, the map
  \begin{equation}
    (\ve,z) \mapsto r_\ve(z)=d_\ve(x,z)
  \end{equation}
  is smooth for $(\ve,z)\in [0, \ve')\times V$. In particular, we have
  uniform convergence $r_\ve \to r_0$ together with their derivatives
  of arbitrary order on compact subsets of $\M \setminus \Cut_0(x)$.
\end{theorem}

\subsection{Sub-Laplacian comparison theorems for the Riemannian
  approximations}

The Riemannian gradient will be denoted by $\nabla$ and the Riemannian volume by $\mu$ (that is, for $\ve=1$) and we write the horizontal gradient as $\nabla_\Ho$, which is the projection of
$\nabla$ onto $\mathcal{H}$. The horizontal (or sometimes called
sub-Riemannian) Laplacian $\Delta_\Ho$ is the generator of the
symmetric closable bilinear form:
\[
  \mathcal{E}_{\mathcal{H}} (f,g) =-\int_\bM \langle
  \nabla_\mathcal{H} f , \nabla_\mathcal{H} g \rangle_{\mathcal{H}}
  \,d\mu, \quad f,g \in C_0^\infty(\M).
\]
(Note that using the volume of any of the Riemannian structures for $\ve>0$ would give the same form and thus the same generator.) The hypothesis that $\mathcal{H}$ is bracket generating implies that
the horizontal Laplacian $\Delta_{\mathcal{H}}$ is locally subelliptic
and the completeness assumption on $g$ implies that
$\Delta_{\mathcal{H}}$ is essentially self-adjoint on the space of
smooth and compactly supported functions (see for instance~\cite{B2}).
Estimates on $\Delta_\ch r_\ve$ outside of the cut-locus have been
obtained in \cite{BGKT} and rely on the control of some tensors
associated to a canonical connection (the Bott connection). The exact
definition of those tensors is not relevant in the present paper, so
for conciseness we omit the details, but refer to \cite{BGKT}. Such
tensors were denoted $\mathbf{Ric}_\Ho$, $\mathbf{J}^2$, and
$\mathbf{Tr} (J^2_{Z})$. Throughout the Sections 2,\,3 and 4 we assume
that globally on $\M$, for every $X \in \Gamma^\infty(\mathcal{H})$
and $Z \in \Gamma^\infty(\mathcal{V})$,
\begin{align}\label{assumptions curvature}
  \mathbf{Ric}_\Ho (X,X) \ge  \rho_1 (r_\varepsilon )  \| X \|^2_{\mathcal{H}},
  \quad -\langle \mathbf{J}^2 X, X  \rangle_\mathcal{H} \le \kappa (r_\varepsilon) \| X \|^2_\mathcal{H},
  \quad -\frac{1}{4} \mathbf{Tr} (J^2_{Z})\ge \rho_2(r_\varepsilon) \| Z \|^2_\mathcal{V},
\end{align}
for some continuous functions $\rho_1,\rho_2,\kappa$ with $\kappa >0$ and $\rho_2 \geq 0$. We moreover
always assume that the foliation is of Yang-Mills type (see
\cite{BGKT}).  The main results obtained in \cite{BGKT} under those
assumptions are the following:

\begin{theorem}[\cite{BGKT}]\label{comparison 1}
  Let $x_0\in \M$ be fixed and for $\varepsilon >0$ let
  $r_\varepsilon (x) =d_\varepsilon (x_0,x).$ Let $x \in \M$,
  $x \neq x_0$ and $x$ not in the $d_\varepsilon$ cut-locus of
  $x_0$. Let $G:[0,r_\varepsilon(x)]\to \mathbb{R}_{\ge 0}$ be a
  differentiable function which is positive on $(0,r_\varepsilon(x)]$
  and such that $ G(0)=0$. We have
  \begin{align*}
    \Delta_\Ho r_\varepsilon(x) \le \frac{1}{G(r_\varepsilon(x))^2} \int_0^{r_\varepsilon(x)} \bigg(nG'(s)^2-\bigg[
    &\bigg(\rho_1 (s )-\frac{1}{\varepsilon} \kappa(s)  \bigg)\Gamma(r_\varepsilon)(x)\\
    &+\rho_2(s)  \Gamma^\V(r_\varepsilon)(x)\bigg]G(s)^2 \bigg)ds.
  \end{align*}
\end{theorem}

\begin{corollary}[\cite{BGKT}]\label{comparison 2}

  Assume that the functions $\rho_1,\kappa,\rho_2$ are constant.
  Denote
  \[
    \kappa_\varepsilon= \min \left(\rho_1-\frac{\kappa}{\varepsilon},
      \frac{\rho_2} {\varepsilon} \right).
  \]
  For $x\neq x_0 \in \M $, not in the $d_\varepsilon$ cut-locus of
  $x_0$,
  \begin{equation}\label{lc2}
    \Delta_{\mathcal{H}} r_{\varepsilon}(x) \le \begin{cases}
      \sqrt {n\kappa_\varepsilon} \cot(\sqrt {\frac{\kappa_\varepsilon}{n}} r_{\varepsilon}(x)),
      &\text{if}\ \kappa_{\varepsilon} >0,
      \\
      \displaystyle\frac{n}{r_{\varepsilon}(x)},&\text{if}\ \kappa_{\varepsilon} = 0,
      \\
      \sqrt{n|\kappa_\varepsilon|} \coth(\sqrt{\frac{|\kappa_\varepsilon|}{n}} r_{\varepsilon}(x)),
      &\text{if}\ \kappa_{\varepsilon} <0.
    \end{cases}
  \end{equation}
\end{corollary}

\section{It\^o's formula for radial processes} \label{sec:RadialProcess}

Let $( ( \xi_t )_{t \ge 0} , ( \P_x )_{x \in \M} )$ be the subelliptic
diffusion process generated by $\Delta_{\ch}$ and let $\zeta$ denote its
lifetime. We will refer to~$\xi$
as the horizontal Brownian motion of the foliation or as the
sub-Riemannian Brownian motion (in particular, here our Brownian motion is normalized to have $\Delta_{\ch}$ as its generator, rather than $\frac{1}{2}\Delta_{\ch}$). Note that $\xi$ admits a smooth heat
kernel $p_t (x,y)$ by the hypoellipticity of $\Delta_{\ch}$.  Take
$x_0 \in \M$ and set $r_\ve (x) := d_{\ve} ( x_0 , x )$.  We denote
the open $g_\ve$-metric ball of radius $r$ centered at $x$ by
$B^\ep_r (x)$, where $\ep \in [ 0 , \infty )$. The goal of this subsection is to show the following It\^o formula for
the radial processes $r_\ve ( \xi_t )$:

\begin{theorem} \label{th:Ito-radial} For each $x_1 \in \M$, if $\xi_0 = x_1$, then there
  exists a non-decreasing continuous process $l_t$ which increases only
  when $\xi_t \in \Cut_\ve (x_0)$ and a martingale $\be_t$ on $\R$
  with quadratic variation $\ang{\be}_t \le 2 t$ such that
  \begin{equation} \label{eq:Ito-radial} r_\ve ( \xi_{ t \wedge \zeta
    } ) = r_\ve ( x_1 ) + \be_t + \int_0^{ t \wedge \zeta }
    \Delta_{\ch} r_\ve ( \xi_s ) ds - l_{ t \wedge \zeta }
  \end{equation}
  holds $\P_{x_1}$-almost surely.
\end{theorem}

We begin the proof with some preparatory lemmas.  The following is the
usual It\^o formula for a smooth function in a local chart.  Let $U$
be an open local chart of $\M$ in which we have
$\Delta_{\ch} = \sum_{i=1}^n X_i^2 + X_0$ with a family of vector
fields $X_0 , X_1 , \ldots , X_n$ on $U$, and let $\xi_t$ satisfy the
stochastic differential equation
\begin{equation} \label{eq:SDE} d \xi_t = \sum_{i=1}^n \sqrt{2} X_i (
  \xi_t ) \circ d W^i_t + X_0 ( \xi_t ) dt,
\end{equation}
where $( W^1_t , \ldots , W^n_t )$ is a Brownian motion on $\R^n$ (here $( W^1_t , \ldots , W^n_t )$ is a standard Euclidean Brownian motion, generated by $\frac{1}{2}$ the Laplacian, which explains the factors of $\sqrt{2}$ in the SDE).
Let $\tau$ be the first exit time from $U$ of~$\xi_t$,
i.e.~$\tau : = \inf \{ t \ge 0 \; | \; \xi_t \notin U \}$.

\begin{lemma} \label{lem:Ito1} For any $U$-valued random variable $S$
 independent of $W$ and smooth function $f : U \to \R$, we have
  \[
    f ( \xi_{t\wedge\tau} ) = f ( S ) + \sum_{i=1}^n \sqrt{2} \int_0^{t\wedge\tau} X_i f (
    \xi_s ) dW^i_s + \int_0^{t\wedge\tau} \Delta_{\ch} f ( \xi_s ) ds,
  \]
  $\P_S$-almost surely.
\end{lemma}

Next we show the following two auxiliary lemmas which concern the
occupation time of $\xi_t$ at singular points of $r_\ve$.  The proof
is almost the same as the one for Riemannian manifolds, but we give it
for completeness.

\begin{lemma} \label{lem:neg_cutlocus} For $\ve>0$, the set
  $ \{ t \in [0, \infty) \, | \, \xi_t \in \Cut_\ve (x_0)\} $ has
  Lebesgue measure zero $\P_{x_1}$-almost surely.
\end{lemma}

\begin{proof}
  Since there is a heat kernel $p_t$, the law of $\xi_t$ under
  $\P_{x_1}$ is absolutely continuous with respect to $\mu$.  In
  addition, we have $\mu ( \Cut_\ve (x_0) ) = 0$ (see \cite{Chavel},
  for instance).  By combining these facts with Fubini's theorem, we
  obtain
  \[
    \E_x \edg{ \int_0^\infty 1_{\{ \xi_s \in \Cut_\ve (x_0) \}} ds } =
    \int_0^\infty \P_x \edg{ \xi_s \in \Cut_\ve (x_0) } ds = 0.
  \]
  Hence $\ds \int_0^\infty 1_{\{ \xi_s \in \Cut_\ve (x_0) \}} ds = 0$
  $\P_{x_1}$-almost surely, and this is nothing but the conclusion.
\end{proof}

\begin{lemma} \label{lem:neg_origin}
  $\P_{x_0} \{ \xi_t = x_0 \mbox{ at some $t \in ( 0, \infty )$} \} =
  0$.
\end{lemma}

\begin{proof}
  We begin with noting that by definition the horizontal Laplacian
  $\Delta_\Ho$ is the generator of the Dirichlet form:
  \[
    \mathcal{E}_{\mathcal{H}} (f,g) =-\int_\bM \langle
    \nabla_\mathcal{H} f , \nabla_\mathcal{H} g \rangle_{\mathcal{H}}
    \,d\mu, \quad f,g \in \mathcal{D} ( \mathcal{E}_\ch ) .
  \]

Observe also that
  $\mu ( B^0_r ( x_0 ) ) \le \mu ( B^\ep_r (x_0) ) \le C_\ep r^n$ for
  small $r > 0$.  This yields that $\{ x_0 \}$ is exceptional by
  applying \cite[Theorem~3]{St0}.  The assumption of the theorem in
  \cite{St0} is satisfied since the Brownian motion $\xi$ is
  associated with the Dirichlet form
  $( \mathcal{E}_\ch , \mathcal{D} ( \mathcal{E}_\ch ) )$ and the
  distance $d_0$ coincides with the intrinsic distance associated with
  $\mathcal{E}_\ch$.  Then, by \cite[Theorem~4.1.2]{FOT2} and
  \cite[Lemma~4.2.4]{FOT2}, $\{ x_0 \}$ is polar. Thus the claim
  holds.
\end{proof}

For $\mathbf{R} = (R_1 , R_2 )$ with $R_1 > R^{-1}_2 > 0$, we define
stopping times $T^{(i)}_{R_i}$ ($i=1,2$) and $T_{\mathbf{R}}$ by
\begin{align*}
  T^{(1)}_{R_1} 
  & : = 
    \inf \crl{ 
    t \ge 0 
    \; | \; 
    r_\ve ( \xi_t ) \ge R_1 
    },
  \\
  T^{(2)}_{R_2} 
  & : = 
    \inf \crl{ 
    t \ge 0 
    \; \left| \; 
    r_\ve ( \xi_t ) \le {1}/{R_2} 
    \right.
    }, 
  \\
  T_{\mathbf{R}} 
  & := 
    T^{(1)}_{R_1} \wedge T^{(2)}_{R_2} , 
\end{align*}
where $a \wedge b = \min \{ a, b \}$ for
$a, b \in \R \cup \{ \pm \infty \}$.  We take $R_2$ sufficiently large so that
$d_\ve ( x_0 , \Cut_\ve (x_0) ) > R_2^{-1}$ holds, and from now on, we fix $R_1$ and
$R_2$ until the final part of the proof of Theorem
\ref{th:Ito-radial}.  Let us define a
set $A$ by
\[
  A := \left\{ ( x , y ) \in \M \times \M \, \left| \,
      \begin{array}{l}
        d_\ve ( x_0 , x ) \in [ R_2^{-1} , R_1 ], \, 
        d_\ve (x_0,y) = ( 3 R_2 )^{-1}
        \\
        \mbox{ and }
        d_\ve (x,y) = d_\ve (x_0 , x) - d_\ve (x_0 , y)
      \end{array}
    \right.  \right\}.
\]
Note that $A$ is compact since $d_\ve (x,y)$ is continuous as a
function of $x$ and $y$.  For $(x,y) \in A$, $y$ is on a minimal
geodesic joining $x_0$ and $x$.  In addition,
$A \cap \Cut_\ve(\M) = \emptyset$ holds since we can extend the minimal
geodesic from $x$ to $y$ with keeping its minimality.  By combining
these facts, we conclude
\[
  \delta_1 : = \inf \crl{ d_\ve ( x, x' ) + d_\ve ( y, y' ) \mid ( x,
    y ) \in A , \, ( x' , y' ) \in \Cut_\ve(\M) } \wedge \frac{1}{ 3 R_2 }
  > 0.
\]
Since we can take $\rho_1 , \rho_2 , \kappa$ to be constants on
$B^\ve_{ R_1 + ( 3 R_2 )^{-1} } (x_0)$, Corollary \ref{comparison 2}
yields that there is a continuous function
$V : ( 0 , R_1 + ( 3 R_2 )^{-1} ) \to [ 0, \infty )$ such that
\begin{equation} \label{eq:dominant} \Delta_{\ch} d_\ve ( x , \cdot )
  (y) \le V ( d_\ve ( x, y ) )
\end{equation}
holds for $x \in B^\ve_{ ( 3 R_2 )^{-1} } (x_0)$ and
$y \in B^\ve_{ R_1 } (x_0) \setminus \Cut_\ve (x)$.  Set $ \bar{V} : =
\sup_{( 3 R_2 )^{-1} \le r \le R_1} V ( r ) $.

\begin{lemma}\label{lem:fund}
  Let $x \in \Cut_\ve (x_0) \cap B_{R_1}^\ve (x_0)$ and
  $\delta \in (0, \delta_1 )$. Set
$$\tilde{T} := \inf \{ t \ge 0 \, | \, d_\ve (x , \xi_t ) \ge \delta \}.$$
Then
\[
  \E_{x} \left[ d_\ve ( x_0 , \xi_{ t \wedge \tilde{T} \wedge
      T_{\mathbf{R}} } ) - d_\ve ( x_0 , x ) - ( t \wedge \tilde{T}
    \wedge T_{\mathbf{R}} ) \bar{V} \right] \le 0.
\]
\end{lemma}

\begin{proof}
  We choose a point $\tilde{x}_0 \in \M$ as follows: Take a minimal
  geodesic $\gamma : [ 0 , r_\ve ( x ) ] \to \M$ from $x_0$ to $x$ and
  define $\tilde{x}_0 := \gamma( ( 3 R_2 )^{-1} )$.  Then
  $( x_0 , \tilde{x}_0 ) \in A$ holds by construction.  Moreover, by
  the choice of $\delta > 0$, $\xi_t \notin \Cut_\ve (\tilde{x}_0)$
  for all $t \in [0, \tilde{T}\wedge T_{\mathbf{R}} )$ under $\P_{x}$.
  For $y \in \M$, let
  \[
    \tilde{r}^+ (y) := d_\ve ( x_0 , \tilde{x}_0 ) + d_\ve
    (\tilde{x}_0, y).
  \]
  By the choice of $\tilde{x}_0$, we have
  $\tilde{r}^+ (x) = d_\ve ( x_0 , x )$.  Moreover, by the triangle
  inequality, $\tilde{r}^+ (y) \ge d_\ve (x_0 , y)$ for all
  $y \in \M$.  By the definition of $V$
  we have
  \[
    \Delta_{\ch} \tilde{r}^+ (y) \le V ( d_\ve ( \tilde{x}_0 , y ) )
  \]
  holds for $y \in B_{R_1} (x_0) \setminus \Cut_\ve ( \tilde{x}_0 )$.
  Note that $V ( d_\ve ( \tilde{x}_0 , \xi_t ) ) \le \bar{V}$ holds
  for all $t \in [0, \tilde{T} \wedge T_{\mathbf{R}} )$ since we have
  \[
    \frac{1}{3 R_2} \le d_\ve ( x_0 , \xi_t ) - d_\ve ( x_0 ,
    \tilde{x}_0 ) \le d_\ve ( \tilde{x}_0 , \xi_t ) \le d_\ve (
    \tilde{x}_0 , x ) + d_\ve ( x , \xi_t ) \le R_1
  \]
  by the choice of $R_2$ and $\delta_1$.  Therefore
  \begin{align*}
    d_\ve ( x_0 , \xi_{ t \wedge \tilde{T} \wedge T_{\mathbf{R}} } ) 
    & 
      - d_\ve ( x_0 , x  ) 
      - ( t \wedge \tilde{T} \wedge T_{\mathbf{R}} ) \bar{V} 
    \\
    & \le 
      d_\ve ( x_0, \xi_{ t \wedge \tilde{T} \wedge T_{\mathbf{R}} } ) 
      - \tilde{r}^+( \xi_0 ) 
      - 
      \int_{0}^{t \wedge \tilde{T} \wedge T_{\mathbf{R}} } 
      V ( d_\ve ( \tilde{x}_0 , \xi_s ) ) 
      ds
    \\
    & \le 
      \tilde{r}^+ ( \xi_{ t \wedge \tilde{T} \wedge T_{\mathbf{R}} } ) 
      - \tilde{r}^+ ( \xi_0 )
      - 
      \int_{0}^{t \wedge \tilde{T} \wedge T_{\mathbf{R}} } 
      \Delta_{\ch} \tilde{r}^+ ( \xi_s ) 
      ds.
  \end{align*}
  Since $\tilde{r}^+$ is smooth on $B^\ve_\delta (x)$, the last term
  is a martingale and thus its expectation is zero.  Hence the claim
  follows.
\end{proof}

For $\delta \in (0 , \delta_1)$, we define a sequence of stopping
times $(S_k^\delta)_{m \in \mathbb{N}}$ and
$(T_k^\delta)_{m \in \mathbb{N}_0}$ by
\begin{align*}
  T_0^\delta 
  & :=  
    0,
  \\
  S_k^\delta 
  & :=  
    T_{\mathbf{R}} \wedge 
    \inf \{ 
    t \ge T_{k-1}^\delta 
    \mid
    \xi_t \in \Cut_\ve (x_0) 
    \},
  \\
  T_k^\delta 
  & :=  
    T_{\mathbf{R}} \wedge 
    \inf \{ 
    t \ge S_k^\delta 
    \mid 
    d_\ve ( \xi_{ S_k^\delta } , \xi_t ) \ge \delta
    \}
    \wedge
    ( S_k^\delta + \delta ). 
\end{align*}
\begin{proposition}\label{prop:super}
  $ r_\ve( \xi_{ t \wedge T_{\mathbf{R}} } ) - r_\ve ( \xi_0 ) - ( t
  \wedge T_{\mathbf{R}} ) \bar{V} $ is a supermartingale.
\end{proposition}

\begin{proof} 
  By virtue of the strong Markov property of $\xi$, it suffices to
  show
$$
\E_x \left[ r_\ve ( \xi_{t \wedge T_{\mathbf{R}} } ) - r_\ve ( \xi_{0}
  ) - ( t \wedge T_{\mathbf{R}} ) \bar{V} \right] \le 0
$$
for each $0 \le s < t$.  By Lemma~\ref{lem:fund}, for all
$n \in \mathbb{N}$
\[
  \E_x \left[ r_\ve ( \xi_{ t \wedge T_k^\delta } ) - r_\ve ( \xi_{ t
      \wedge S_{k}^\delta } ) - \brak{ t \wedge T_k^\delta - t \wedge
      S_{k}^\delta } \bar{V} \, \Big| \, \F_{S_{k}^\delta} \right] \le
  0.
\]
We apply the It\^o formula to $r_\ve ( \xi_t )$ on
$t \in [ T_{k-1}^\delta , S_k^\delta ]$.  By \eqref{eq:dominant}, we
have
\begin{equation*}
  \Delta_{\ch} r_\ve (\xi_{ t \wedge S^\delta_k }) 
  \le 
  \bar{V}
\end{equation*}
for $t > T_{k-1}^\delta$.  These observations yield
\[
  \E_x \left[ r_\ve ( \xi_{ t \wedge S_k^\delta } ) - r_\ve ( \xi_{ t
      \wedge T_{k-1}^\delta } ) - \brak{ t \wedge S_k^\delta - t
      \wedge T_{k-1}^\delta } \bar{V} \, \Big| \, \F_{T_{k-1}^\delta}
  \right] \le 0.
\]
It remains to show $T_k^\delta \to T_{\mathbf{R}}$ as $k \to \infty$
in order to conclude the claim by the dominated convergence theorem.
If $\lim_{k \to \infty } T_k^\delta =: T_\infty < T_{\mathbf{R}}$
occurs, then $T_k^\delta - S_k^\delta > 0$ converges to 0 as
$k \to \infty$.  In addition,
$d ( \xi_{ S_k^\delta } , \xi_{ T_k^\delta } ) = \delta$ must hold for
infinitely many $k \in \mathbb{N}$.  However, the combination of these
contradicts the fact that the sample path $\xi_t$ is uniformly
continuous on $[ 0, T_\infty ]$.  Hence
$\ds \lim_{k \to \infty} T_k^\delta = T_{\mathbf{R}}$ as
$k \to \infty$.
\end{proof}

\begin{corollary} \label{cor:semi}
  $r_\ve ( \xi_{ t \wedge T_{\mathbf{R}} } )$ is a semimartingale.
\end{corollary}

\begin{remark} \label{rem:time-diverge} Repeating the last argument in
  the proof of Proposition~\ref{prop:super} implies that, for each
  fixed $t > 0$, $t \wedge T_{k}^\delta = t \wedge T_{\mathbf{R}}$
  holds for sufficiently large $k$ almost surely.
\end{remark}

\begin{lemma} \label{lem:null}
  $ \lim_{\delta \to 0} \sum_{k=1}^\infty | T_k^\delta - S_k^\delta |
  = 0 $ almost surely.
\end{lemma}

\begin{proof}
  For $\delta > 0$, let us define random subsets $H_\delta$ and $H$ of
  $[0, T_{\mathbf{R}} )$ by
  \begin{align*}
    H_\delta 
    & : = 
      \left\{ 
      t \in [ 0 , T_{\mathbf{R}} ) 
      \, | \, 
      \mbox{there exists $t' \in [ 0, T_{\mathbf{R}} )$ satisfying }
      |t - t'| \le \delta 
      \mbox{ and } 
      \xi_{t'} \in \Cut_\ve (x_0)
      \right\} ,  
    \\
    H 
    & : = 
      \left\{ 
      t \in [ 0 , T_{\mathbf{R}} ) 
      \, | \, 
      \xi_t \in \Cut_\ve (x_0)
      \right\} . 
  \end{align*}  
  Since the sample path of $\xi$ is continuous and $\Cut_\ve (x_0)$ is
  closed, $H$ is closed and $H = \cap_{ \delta > 0 } H_\delta$ holds.
  By the definition of $S_k^\delta$ and $T_k^\delta$, we have
  \[
    H \subset \bigcup_{m=1}^\infty [ S_k^\delta , T_k^\delta ] \subset
    H_\delta .
  \]
  Hence the monotone convergence theorem yields that, for any $T > 0$,
  \[
    \limsup_{\delta \to 0} \sum_{m=1}^\infty | T_k^\delta \wedge T -
    S_k^\delta \wedge T | \le \lim_{\delta \to 0} \int_0^{T}
    1_{H_\delta} (t) dt = \int_0^{T} 1_H (t) dt = 0
  \]
  almost surely, where the last equality follows from
  Lemma~\ref{lem:neg_cutlocus}.
\end{proof}

\begin{lemma} \label{lem:martingale} Let $U$ be a local chart of $\M$
  on which $\Delta_{\ch} = \sum_{i=1}^n X_i^2 + X_0$ holds with a
  family of vector fields $\{ X_i \}_{i=0}^n$ on $U$ and $\xi_t$
  satisfies the stochastic differential equation \eqref{eq:SDE} with a
  Brownian motion $( W^1_t , \ldots , W^n_t )$.  Let $\tau_1$ and
  $\tau_2$ be stopping times with $\tau_1 < \tau_2$ so that $\xi_t$ is
  in $U$ whenever $t \in [ \tau_1 , \tau_2 ]$.  Then the martingale
  part of
  $ r_\ve ( \xi_{ t \wedge T_{\mathbf{R}} \wedge \tau_2 } ) - r_\ve (
  \xi_{ T_{\mathbf{R}} \wedge \tau_1 } ) $ coincides with
  \begin{equation} \label{eq:mart} \sqrt{2} \sum_{i=1}^n
    \int_{T_{\mathbf{R}} \wedge \tau_1}^{t \wedge T_{\mathbf{R}}
      \wedge \tau_2} X_i r_\ve ( \xi_s ) d W_s^i .
  \end{equation}
  Moreover, the quadratic variation of the martingale part of
  $r_\ve ( \xi_{ t \wedge T_{\mathbf{R}} } ) - r_\ve ( \xi_0 )$ is
  bounded from above by $2t$.
\end{lemma}

\begin{proof}
  We first remark that the integrand $X_i r_\ve ( \xi_s )$ of the
  It\^o stochastic integral \eqref{eq:mart} is well-defined by virtue
  of Lemma~\ref{lem:neg_cutlocus}.  Moreover,
  $\sum_{i=1}^n | X_i r_\ve ( \xi_s ) |^2 \le 1$ holds for a.e.~$s$
  $\P_{x_1}$-almost surely.  By the martingale representation theorem,
  there exists an $\R^n$-valued adapted process $\eta$ such that the
  martingale part of
  $ r_\ve ( \xi_{ t \wedge T_{\mathbf{R}} \wedge \tau_2 } ) - r_\ve (
  \xi_{ T_{\mathbf{R}} \wedge \tau_1 } ) $ equals
$$\sum_{i=1}^n \int_{\tau_1 \wedge T_{\mathbf{R}}}^{t \wedge T_{\mathbf{R}} \wedge \tau_2 } \eta^i_s dW_s^i.$$
Let us define a (local) martingale $N_t$ by
\[
  N_t := \sum_{i=1}^n \int_{T_{\mathbf{R}} \wedge \tau_1}^{ t \wedge
    T_{\mathbf{R}} \wedge \tau_2 } \eta^i_s dW_s^i - \sqrt{2}
  \sum_{i=1}^n \int_{T_{\mathbf{R}} \wedge \tau_1}^{t \wedge
    T_{\mathbf{R}} \wedge \tau_2} X_i r_\ve ( \xi_s ) d W_s^i .
\]
Using the stopping times $S_k^\delta$ and $T_k^\delta$, the quadratic
variation $\langle N \rangle_t$ of $N$ is expressed as follows:
\begin{align} \label{eq:quadratic} \langle N \rangle_{\tau_2} =
  \sum_{i=1}^d \sum_{k=1}^\infty \Bigg( &\int_{T_{k-1}^\delta \wedge
    T_{\mathbf{R}} \wedge \tau_1}^{ S_k^\delta \wedge T_{\mathbf{R}}
    \wedge \tau_2 } | \eta^i (t) - \sqrt{2} X_i r_\ve ( \xi_t) |^2
  dt\notag
  \\
                                        &+ \int_{S_k^\delta \wedge
                                          T_{\mathbf{R}} \wedge \tau_1
                                          }^{T_k^\delta \wedge
                                          T_{\mathbf{R}} \wedge
                                          \tau_2} | \eta^i_t -
                                          \sqrt{2} X_i r_\ve ( \xi_t )
                                          |^2 dt \Bigg).
\end{align}  

Since $\xi_t \notin \Cut_\ve (x_0)$ if
$t \in (T_{k-1}^\delta , S_k^\delta )$, Lemma \ref{lem:Ito1} yields
\[
  \int_{T_{k-1}^\delta \wedge T_{\mathbf{R}} \wedge \tau_1}^{
    S_k^\delta \wedge T_{\mathbf{R}} \wedge \tau_2} | \eta^i_t -
  \sqrt{2} X_i r_\ve ( \xi_t ) |^2 dt = 0
\]
for $k \in \mathbb{N}$ and $i =1 , \ldots , n$.  For the second term
in the right-hand side of \eqref{eq:quadratic} we have
\begin{equation} \label{eq:mart_neg} \sum_{k=1}^\infty
  \int_{S_k^\delta \wedge T_{\mathbf{R}} \wedge \tau_1 }^{T_k^\delta
    \wedge T_{\mathbf{R}} \wedge \tau_2} | \eta^i_t - \sqrt{2} X_i
  r_\ve ( \xi_t ) |^2 dt \le 2 \int_{ \bigcup_{k=1}^\infty [
    S_k^\delta \wedge T_{\mathbf{R}} \wedge \tau_1 , T_k^\delta \wedge
    T_{\mathbf{R}} \wedge \tau_2 ] } \left( | \eta_t |^2 + 2 \right)
  dt .
\end{equation}
Since $\eta$ is locally square-integrable on $[0 ,\infty)$ almost
surely, Lemma~\ref{lem:null} yields that the right hand side of
\eqref{eq:mart_neg} tends to 0 as $\delta \downarrow 0$.  Hence
$\langle N \rangle_{\tau_2} = 0$ and the first assertion follows.  The
second assertion can be obtained by decomposing $\xi_t$ through a
sequence of stopping times into sample paths each of which is
contained in a local chart and using the strong Markov property.
\end{proof}

We are final ready to prove our It\^o formula for the radial process.
 
\begin{proof}[Proof of Theorem~\ref{th:Ito-radial}]
  On the basis of Corollary \ref{cor:semi}, we denote the martingale
  part of $r_\ve ( \xi_{ t \wedge T_{\mathbf{R}}} )$ by
  $\beta^{\mathbf{R}}_{ t }$.
  By virtue of \eqref{eq:dominant} and Lemma \ref{lem:neg_cutlocus},
  the integral of $s \mapsto \Delta_{\ch} r_\ve ( \xi_s )$ on a subset
  of $[0,t \wedge T_{\mathbf{R}} ]$ is well-defined.  Set
  $ I_\delta : = \bigcup_{n=1}^\infty [ S_k^\delta , T_k^\delta ] $
  and let us define $l^{\delta , \mathbf{R}}_t$ by
  \begin{align*}
    l^{\delta, \mathbf{R}}_t
    & := 
      - r_\ve ( \xi_{ t \wedge T_{\mathbf{R}} } ) + r_\ve ( \xi_0 ) 
      + 
      \beta^{\mathbf{R}}_{ t \wedge T_{\mathbf{R}} } 
    \\
    & \qquad 
      +
      \int_{ [ 0, t \wedge T_{\mathbf{R}} ] \setminus I_\delta }
      \Delta_{\ch} r_\ve ( \xi_s ) 
      ds
      +  
      | [0, t \wedge T_{\mathbf{R}} ] \cap I_\delta |\bar{V}, 
  \end{align*}
  where the modulus indicates the Lebesgue measure of the set.  By
  Lemmas~\ref{lem:Ito1} and \ref{lem:martingale},
  $l^{\delta ,\mathbf{R}}$ is constant on
  $[0, T_{\mathbf{R}} ] \setminus I_\delta$.  Moreover,
  Proposition~\ref{prop:super} yields that
  $l^{ \delta , \mathbf{R}}_t$ is non-decreasing in $t$, and in
  particular $l^{\delta, \mathbf{R}}_t \ge 0$.  By Lemma
  \ref{lem:neg_cutlocus},
  $\Delta_{\ch} r_\ve ( \xi_s ) 1_{I_\delta^c} ( \xi_s )$ converges to
  $\Delta_{\ch} r_\ve ( \xi_s )$ as $\delta \downarrow 0$ for
  a.e.~$s \in [ 0 , t \wedge T_{\mathbf{R}} ]$ $\P_{x_1}$-a.s.  Thus,
  the Fatou lemma together with \eqref{eq:dominant} and Lemma
  \ref{lem:null} yields
  \begin{align}
    \nonumber
    0 \le \limsup_{\delta \downarrow 0} l^{\delta, \mathbf{R}}_t 
    & \le
      - r_\ve ( \xi_{ t \wedge T_{\mathbf{R}} } )  
      + r_\ve ( \xi_0 ) 
      + \beta^{\mathbf{R}}_{ t \wedge T_{\mathbf{R}} } 
      + \int_0^{ t \wedge T_{\mathbf{R}} } 
      \Delta_{\ch} r_\ve ( \xi_s ) 
      ds .
  \end{align}
  This inequality and \eqref{eq:dominant} ensure that
  $s \mapsto \Delta_{\ch} r_\ve ( \xi_s )$ belongs to
  $L^1 ( [ 0, t \wedge T_{\mathbf{R}} ] )$ $\P_{x_1}$-a.s.  Thus,
  Lemma \ref{lem:null} implies that
  $l^{\mathbf{R}}_t := \lim_{\delta \downarrow 0} l^{\delta,
    \mathbf{R}}_t$ exists and
  \[
    r_\ve ( \xi_{ t \wedge T_{\mathbf{R}} } ) = r_\ve ( \xi_0 ) +
    \beta^{\mathbf{R}}_{ t \wedge T_{\mathbf{R}} } + \int_0^{ t \wedge
      T_{\mathbf{R}} } \Delta_{\ch} r_\ve ( \xi_s ) ds -
    l^{\mathbf{R}}_t
  \]
  holds for all $t \ge 0$ $\P_{x_1}$-a.s.  By Lemma
  \ref{lem:neg_origin}, we can take the limit $R_2 \to \infty$ in the
  last identity in a compatible way. Then the conclusion follows by
  taking the limit $R_1 \to \infty$.  Indeed, it is not difficult to
  see that $l_t$ can increase only when $\xi_t \in \Cut_\ve (x_0)$
  from the corresponding property for $l^{\delta, \mathbf{R}}_t$.
\end{proof}

\section{ Comparison of radial processes and stochastic completeness
  on general foliations} \label{sec:ComparisonGen}

\subsection{Comparison of radial processes}

We first recall the definition of a model Riemannian manifold that was
introduced by E.~Greene and H.~Wu, see \cite{GW}.

\begin{definition}
  Let $h\colon[0,+\infty)\to (0,+\infty)$ be a smooth function which is
  positive on $(0,+\infty)$ and such that $h(0)=0$, $h'(0)=1$. Denote
  $K(r)=-{h''(r)}/{h(r)}$.  The Riemannian manifold
  $\M_K=[0,+\infty) \times \mathbb{S}^n$ with Riemannian metric
  \[
    g =dr^2+h(r)^2 g_{\mathbb{S}^n}
  \]
  is called the Riemannian model space with radial curvature $K(r)$, where
  $g_{\mathbb{S}^n}$ denotes the standard metric on $\mathbb{S}^n$.
\end{definition}

As before, we fix a point $x_0 \in \M$. For $x_1 \in \M$, we consider
the solution of the stochastic differential equation
\[
  \tilde{\xi}_t=d_\ve (x_0,x_1)+n\int_0^t
  \frac{h'(\tilde{\xi}_s)}{h(\tilde{\xi}_s)} ds+\beta_t
\]
where $\beta$ is the martingale defined in Theorem
\ref{th:Ito-radial}.

\begin{theorem}\label{comparison model}
  Assume that
  \[
    \min \left\{ \rho_1 (r) - \frac{ \kappa (r) }{\ep} , \frac{ \rho_2
        (r) }{\ep} \right\} \ge n K(r).
  \]
  Let $( ( \xi_t )_{t \ge 0} , ( \P_x )_{x \in \M} )$ be the
  horizontal Brownian motion on $\M$ generated by $\Delta_{\ch}$.

  Then, for any $R \ge d_\ve(x_0,x_1)$ and any non decreasing function
  $\phi$ on $[0,R)$,
  \[
    \E_{x_1}\big[ \phi( d_\ve (x_0,\xi_t) ) ,\, t < \tau_R \big] \le
    \E_{x_1} \big[ \phi(\tilde{\xi}_t) ,\, t < \tau_R \big],
  \]
  where $\tau_R$ is the hitting time of $R$ by $\tilde \xi$.
\end{theorem}

\begin{proof}
  It follows from Theorem \ref{th:Ito-radial} and the Ikeda-Watanabe
  comparison theorem \cite{Ik-Wa} that for $t < \tau_R$, one has
  $\P_{x_1}$ a.s.
  \[
    d_\ve (x_0,\xi_t) \le \tilde{\xi}_t.
  \]
  The result follows then immediately.
\end{proof}

\subsection{Stochastic completeness criterion}

In this section we prove a general non explosion
criterion for the horizontal Brownian $(\xi_t)_{t \ge 0}$ as a consequence of the sub-Laplacian
comparison Theorem~\ref{comparison 1}. Recall that
the functions $\rho_1,\rho_2,\kappa$ are defined through the
assumptions \eqref{assumptions curvature}.

\begin{theorem} \label{thm:S-cpl} Suppose that there exists $c_1 > 0$
  such that
  \[
    \min \left\{ \rho_1 (s) - \frac{ \kappa (s) }{\ep} , \frac{ \rho_2
        (s) }{\ep} \right\} \ge - n ( c_1^2 s^2 + c_1 )
  \]
  holds for all $s > 0$.  Then $(\xi_t)_{t \ge 0}$ does not explode.
\end{theorem}

\begin{proof}
  Let $\tau = \inf \{ t > 0 \mid r_\ep ( \xi_t ) \le c_1 \}$.  By the
  strong Markov property of $( \xi_t , \P_x )$, it suffices to show
  that $( \xi_{t \wedge \tau} )_{t \ge 0}$ does not explode for each
  $x \in \M$ with $r_\ep (x) > c_1$.  Let us define
  $G : [ 0 , \infty ) \to [ 0 , \infty )$ be a strictly increasing
  $C^2$-function satisfying $G (0) = 0$ and
  $G (s) = \exp ( c_1 s^2 / 2 )$ for $s \ge c_1$.  Note that
  $G'' (s) = ( c_1^2 s^2 + c_1 ) G(s)$ holds.  By Theorem
  \ref{comparison 1}, there exists $C_0 > 0$ such that, for $y \in \M$
  with $r_\ep (y) > c_1$ and $y \notin \Cut_\ep (x_0)$,
  \begin{align*}
    \Delta_\Ho r_\varepsilon(y)
    & \le
      \frac{1}{G(r_\varepsilon(y))^2}
      \int_0^{r_\varepsilon(y)}
      \bigg(
      nG'(s)^2
    \\
    & \hspace{9em}
      -
      \left[
      \left(
      \rho_1 (s)
      -
      \frac{1}{\varepsilon} \kappa(s)
      \right)
      \Gamma(r_\varepsilon)(y)
      +
      \rho_2(s) \Gamma^\V (r_\varepsilon)(y)
      \right]
      G (s)^2
      \bigg)
      ds
    \\
    & \le 
      C_0
      +
      \frac{n}{G(r_\varepsilon(y))^2}
      \int_{c_1}^{r_\varepsilon(y)}
      \left(
      G'(s)^2
      + 
      ( c_1^2 s^2 + c_1 ) G (s)^2
      \right)
      ds
    \\
    & = 
      C_0
      +
      \frac{n}{G(r_\varepsilon(y))^2}
      \left(
      G'( r_\ep (y) ) G ( r_\ep (y) )
      -  
      G' ( c_1 ) G ( c_1 )
      \right)
    \\
    & \le 
      C_0 + n c_1 r_\ep (y). 
  \end{align*}
  Let us define a stochastic process $( \bar{r}_t )_{t \ge 0}$ by
  \[
    \bar{r}_t : = e^{n c_1 t } \left( r_\ep (x) + \int_0^t e^{- n c_1
        s } d \beta_s + \frac{ C_0 ( 1 - e^{ - n c_1 t } ) }{ n c_1 }
    \right) .
  \]
  Note that $\bar{r}$ solves the following stochastic differential
  equation:
  \begin{align*}
    d \bar{r}_t 
    & = 
      d \beta_t + ( n c_1 \bar{r}_t + C_0 ) d t ,
    \\
    \bar{r}_0 & = r_\ep (x). 
  \end{align*}
  By arguing as in Theorem \ref{comparison model}, we can show
  $r_\ep ( \xi_{ t \wedge \tau } ) \le \bar{r}_{ t \wedge \tau}$ for
  $t \in [ 0 , \bar{\zeta} )$ a.s., where
  $ \bar{\zeta} : = \inf \{ t \ge 0 \mid \bar{r}_t \in ( c_1 , \infty
  ) \} $. By the martingale representation as a time-change of a
  Brownian motion, the inequality $\langle \beta \rangle_t \le t$
  implies that $\int_0^t e^{- n c_1 s} d \beta_s$ does not explode in
  a finite time. Hence, the same is true for $\bar{r}_t$.  That is,
  $\bar{r}_{\bar{\zeta}} = c_1$ holds if $\bar{\zeta} < \infty$.
  These facts yield that $r_\ep ( \xi_{t \wedge \tau} )$ does not
  explode.  Thus the conclusion follows.
\end{proof}

\section{Comparison theorems for the radial processes on Sasakian
  manifolds} \label{sec:Sasakian}

In this section we study more in details specific foliations on which
the theory can be pushed further. The foliations we consider are
called Sasakian foliations. Those are well-studied co-dimension one
totally geodesic foliations with additional structure described in
\cite[Section~3]{BGKT} (we also for instance refer to \cite{BD} and the
references inside for further details).

\subsection{Comparison of radial processes}

In this section, we use the sub-Laplacian comparison theorem on
Sasakian manifolds foliations in \cite{BGKT} to get estimates for
radial parts of the horizontal Brownian motion. In the Riemannian case
the method we use is due to K. Ichihara \cite{Ich}.  In the case where
the Ricci curvature is bounded from below by a constant, the method
yields the sharp Cheeger-Yau lower bound \cite{ChY} for the heat
kernel.

\

We first briefly recall the sub-Laplacian comparison theorem proved in
\cite{BGKT} to which we refer for further details. Recall the comparison functions $F_{\mathrm{Rie}}$ and $F_{\mathrm{Sas}}$ given as in respectively \eqref{FRie} and \eqref{FSas}.

\begin{theorem}[\cite{BGKT}] \label{th:SasakianComp3} Let
  $(\M,\mathcal{F},g)$ be a Sasakian foliation with sub-Riemannian
  distance $d_0$. Let $x_0\in \M$ and define $r_0(x) = d(x_0 ,x)$.
  Assume that for some $k_1, k_2 \in \mathbb{R}$
$$\mathbf{K}_{\mathcal{H},J}(v,v) \ge  k_1, \qquad \mathbf{Ric}_{\mathcal{H},J^\perp}(v,v) \ge (n-2)k_2, \qquad v \in \ch,\ \| v\|_g = 1.$$
Then outside of the $d_0$ cut-locus of $x_0$ and globally on $\M$ in
the sense of distributions,
\[
  \Delta_\Ho r_0 \le F_{\mathrm{Sas}}(r,k_1) + (n-2)
  F_{\mathrm{Rie}}(r,k_2).
\]
\end{theorem}

The tensors $\mathbf{K}_{\mathcal{H},J}$ and
$\mathbf{Ric}_{\mathcal{H},J^\perp}$ are defined in \cite{BGKT}. We
omit here their definition for conciseness since they will not be
relevant in our analysis except as criteria to get the sub-Laplacian
comparison theorem.

\

As before, we fix a point $x_0 \in \M$.  For $x_1 \in \M$, we consider
the solution of the stochastic differential equation
\[
  \tilde{\xi}_t=d_0 (x_0,x_1)+\int_0^t
  \left(F_{\mathrm{Sas}}(\tilde{\xi}_s,k_1) + (n-2)
    F_{\mathrm{Rie}}(\tilde{\xi}_s,k_2) \right) ds+\sqrt{2} \beta_t
\]
where $\beta$ is a standard Brownian motion under $\P_{x_1}$.

\begin{theorem}\label{comparison 1d}
  Let $(\M,\mathcal{F},g)$ be a Sasakian foliation with sub-Riemannian
  distance $d_0$.  Assume that for some $k_1, k_2 \in \mathbb{R}$,
$$\mathbf{K}_{\mathcal{H},J}(v,v) \ge  k_1, \quad \mathbf{Ric}_{\mathcal{H},J^\perp}(v,v) \ge (n-2)k_2, \quad v \in \ch, \| v\|_g = 1.$$
Let $( ( \xi_t )_{t \ge 0} , ( \P_x )_{x \in \M} )$ be the horizontal
Brownian motion on $\M$ generated by $\Delta_{\ch}$. Then, for
$x_1 \in \M$, $R>0$, and $s \le R$
\[
  \P_{x_1}\big\{ d_0 (x_0,\xi_t) <s , \, t \le \tau_R \big\} \ge
  \P_{x_1}\big\{ \tilde{\xi}_t<s, \, t \le \tilde{\tau}_R \big\},
\]
where $\tau_R$ is the hitting time of the $d_0$ geodesic ball in $\M$
with center $x_0$ and radius $R$ and $\tilde{\tau}_R$ the hitting time
of the level $R$ by $\tilde \xi$.
\end{theorem}

\begin{proof}
  Let $\phi$ be a non-increasing function on $[0,R]$ which is
  compactly supported on $[0,s]$. We set
  \[
    u(t,x)=\E_{x}\left[ \phi(d_0 (x_0,\xi_t)) , \, t \le \tau_R
    \right]
  \]
  and
  \[
    u_0(t,r)=\E_{r} \big[ \phi( \tilde \xi_t ) ,\, t \le
    \tilde{\tau}_R \big].
  \]
  We then have
  \begin{align*}
    \begin{cases}
      u \in C^\infty ((0,+\infty) \times B_0 (x_0,R) )\\
      \displaystyle\frac{\partial u}{\partial t}=\Delta_\Ho u \\
      u(0,x)=\phi(d_0 (x_0,x)) , \quad u(t,x)=0 \text{ if } x \in
      \partial B_0 (x_0,R)
    \end{cases}
  \end{align*}
  and
  \begin{align*}
    \begin{cases}
      u_0 \in C^\infty ((0,+\infty) \times [0,R) )\\
      \displaystyle\frac{\partial u_0}{\partial t}=L u_0 \\
      u_0(0,r)=\phi(r) , \quad u_0(t,r)=0 \text{ if } r=R,
    \end{cases}
  \end{align*}
  where
  \[
    L= \big(F_{\mathrm{Sas}}(r,k_1) + (n-2)
    F_{\mathrm{Rie}}(r,k_2)\big) \frac{\partial}{\partial r}
    +\frac{\partial^2}{\partial r^2}.
  \]
  Similarly to Lemma 2.1 in \cite{Ich}, $u_0(t,r)$ is non-increasing
  in $r$.
  
  \

  For $t \ge 0, x \in B_0(x_0,R) $ denote now
  $v(t,x)=u_0(t,d_0 (x_0,x))$. For
  $x \in B_0(x_0,R)\setminus \Cut_{x_0}(\M)$, one has then from Theorem
  \ref{th:SasakianComp3},
  \begin{align*}
    &\Delta_\Ho v (t,x)=\frac{\partial^2u_0 }{\partial r^2} (t , d_0 (x_0,x))+\Delta_\Ho r_0 (x) \frac{\partial u_0 }{\partial r} (t , d_0 (x_0,x)) \\
    & \ge \frac{\partial^2u_0 }{\partial r^2} (t , d_0 (x_0,x))+ \big(F_{\mathrm{Sas}}(d_0 (x_0,x),k_1) +    (n-2)  F_{\mathrm{Rie}}(d_0 (x_0,x),k_2) \big)\frac{\partial u_0 }{\partial r} (t , d_0 (x_0,x)) \\
    & \ge Lu_0 (t , d_0 (x_0,x))=\frac{\partial u_0}{\partial t} (t , d_0 (x_0,x))=\frac{\partial v}{\partial t} (t , x).
  \end{align*}
  Therefore, by using the semi-concavity of the sub-Riemannian
  distance and arguing as in the proof of Theorem 10.1 in \cite{AL1},
  one deduces that in the sense of distributions one has for
  $t \ge 0$, $x \in B_0(x_0,R) $,
  \[
    \frac{\partial v}{\partial t} (t , x) \le \Delta_\Ho v (t,x)
  \]
  Since
  \[
    v(0,x)=\phi(d_0 (x_0,x)) , \quad v(t,x)=0\ \text{ if } x \in
    \partial B_0 (x_0,R),
  \]
  a standard parabolic comparison theorem yields
  \[
    v(t,x)\le u(t,x).
  \]
  Taking $\phi = 1$ on $[0,s)$ yields the conclusion.
\end{proof}

As a first corollary we deduce exit time estimates:

\begin{corollary}
  Under the same assumptions as in theorem \ref{comparison 1d} one has
  \[
    \E_{x_1}\left[ \tau_R \right] \ge \E_{x_1} \left[ \tilde{\tau}_R
    \right].
  \]

\end{corollary}

As a second corollary we deduce estimates for the integrals of the
Dirichlet heat kernel on sub-riemannian balls:

\begin{corollary}\label{comparison heat kernel}
  Let $(\M,\mathcal{F},g)$ be a Sasakian foliation. Assume that for
  some $k_1, k_2 \in \mathbb{R}$
$$\mathbf{K}_{\mathcal{H},J}(v,v) \ge  k_1, \quad \mathbf{Ric}_{\mathcal{H},J^\perp}(v,v) \ge (n-2)k_2, \quad v \in \ch,\ \| v\|_g = 1.$$
Let $R>0$ and let
$( ( \xi^R_t )_{t \ge 0} , ( \P_x )_{x \in B_0(x_0,R)} )$ be the
horizontal Brownian motion on $B_0(x_0,R)$ generated by $\Delta_{\ch}$
and with Dirichlet boundary condition. Let $p^R(t,x,y)$ be its heat
kernel with respect to the Riemannian volume measure $\mu$. Let now
$q^R(t,r_1,r_2)$ be the heat kernel with respect to the Lebesgue
measure on $[0,R]$ of the diffusion with generator
\begin{equation}\label{comparison generator}
  L_{k_1,k_2}= \big(F_{\mathrm{Sas}}(r,k_1) +  (n-2)  F_{\mathrm{Rie}}(r,k_2)\big) \frac{\partial}{\partial r} +\frac{\partial^2}{\partial r^2} 
\end{equation}
with Dirichlet boundary condition at $R$. Then, for every $s <R$,
$t>0$ and $x_1\in B_0(x_0,s)$
\[
  \int_{B_0(x_0,s)} p^R(t,x_1,y) d\mu(y) \ge \int_0^s
  q^R(t,d(x_0,x_1),r)dr.
\]

\end{corollary}

It is interesting to note that Corollary \ref{comparison heat kernel}
does not yield a lower bound for the heat kernel $p^R$ as one could
at first expect. Indeed, when $s$ is small the volume of $B_0(x_0,s)$ has
order $s^{n+2}$, because $n+2$ is the Hausdorff dimension of $\M$ for
the metric $d_0$. On the other hand, one can directly check that
$\int_0^s q^R(t,d(x_0,x_1),r)dr$ has order $s^{n+3}$. This discrepancy is
due to the fact that on sub-Riemannian manifolds, the measure
contraction dimension is larger than the Hausdorff dimension.

\subsection{Application to Dirichlet eigenvalue estimates}\label{sec
  Cheng}

We now use the comparison theorem of the previous subsection to deduce
estimates on the first Dirichlet eigenvalue of the sub-Riemannian
balls.  This is the sub-Riemannian version of the well-known Cheng's
comparison theorem in Riemannian geometry.

For simplicity, we start with the non negative curvature case
$k_1=k_2=0$.

Let $(\M,\mathcal{F},g)$ be a Sasakian foliation with sub-Riemannian
distance $d_0$. Assume that:
$$\mathbf{K}_{\mathcal{H},J}(v,v) \ge  0, \quad \mathbf{Ric}_{\mathcal{H},J^\perp}(v,v) \ge 0, \quad v \in \ch,\ \| v\|_g = 1.$$

In that case, the one-dimensional diffusion with respect to which we
do the comparison is very simple since
\[
  L_{0,0}=\frac{n+2}{r} \frac{\partial}{\partial r}
  +\frac{\partial^2}{\partial r^2}
\]
which is a Bessel diffusion of dimension $n+3$. We recall that $n$ is
the dimension of the horizontal distribution.

\begin{theorem}
  Assume $k_1=k_2=0$. For $x_0 \in \M$ and $R>0$, let
  $\lambda_1( B_0(x_0,R))$ denote the first Dirichlet eigenvalue of
  the sub-Riemannian ball $B_0(x_0,R)$ and let
  $\tilde{\lambda}_1(d,R)$ denote the first Dirichlet eigenvalue of
  Euclidean ball with radius $R$ in $\mathbb{R}^d$. Then, for every
  $x_0 \in \M$ and $R>0$
  \[
    \lambda_1( B_0(x_0,R)) \le \tilde{\lambda}_1( n+3,R).
  \]
\end{theorem}

\begin{proof}
  From spectral theory, one has
  \[
    p^R(t,x_1,y) =\sum_{k=1}^{+\infty} e^{-\lambda _k t} \phi_k (x_1)
    \phi_k(y)
  \]
  where the $\lambda_k$'s are the Dirichlet eigenvalues of
  $B_0(x_0,R)$ and the $\phi_k$'s the eigenfunctions.  One has
  similarly
  \[
    q^R(t,r_0,r)=r^{n+2} \sum_{k=1}^{+\infty} e^{-\tilde{\lambda} _k
      t} \tilde{\phi}_k (r_0) \tilde{\phi}_k(r)
  \]
  Thus, from Corollary \ref{comparison heat kernel}, when
  $t \to +\infty$ one must have $\lambda_1 \le \tilde{\lambda}_1$.
\end{proof}

For general $k_1,k_2$ one can similarly prove the following theorem:

\begin{theorem}
  For $x_0 \in \M$ and $R>0$, let $\lambda_1( B_0(x_0,R))$ denote the
  first Dirichlet eigenvalue of the sub-Riemannian ball $B_0(x_0,R)$
  and let $\tilde{\lambda}_1(n, k_1,k_2,R)$ denote the first Dirichlet
  eigenvalue of the operator
  \[
    L_{k_1,k_2}= \big(F_{\mathrm{Sas}}(r,k_1) + (n-2)
    F_{\mathrm{Rie}}(r,k_2)\big) \frac{\partial}{\partial r}
    +\frac{\partial^2}{\partial r^2}
  \]
  on the interval $[0,R]$ with Dirichlet boundary condition at
  $R$. Then, for every $x_0 \in \M$ and $R>0$
  \[
    \lambda_1( B_0(x_0,R)) \le \tilde{\lambda}_1(n, k_1,k_2,R).
  \]
\end{theorem}

\subsection{Large time behavior and law of iterated logarithm for the
  radial processes} \label{sec:Htype}

In this section we study large time behaviour of the radial process in
negative curvature.

\begin{proposition}
  Let $(\M,\mathcal{F},g)$ be a Sasakian foliation. Assume that for
  some $k_1, k_2 \le 0$
$$\mathbf{K}_{\mathcal{H},J}(v,v) \ge  k_1, \quad \mathbf{Ric}_{\mathcal{H},J^\perp}(v,v) \ge (n-2)k_2,
\quad v \in \ch,\ \| v\|_g = 1.$$ Let
$( ( \xi_t )_{t \ge 0} , ( \P_x )_{x \in \M} )$ be the sub-Riemannian
Brownian motion generated by $\Delta_{\ch}$.  Then for every
$x_0,x_1 \in \M$,
\[
  \P_{x_1} \left( \limsup_{t \to +\infty} \frac{d_0(x_0 , \xi_t)}{t}
    \le (n-2) \sqrt{|k_2|} + \sqrt{|k_1|} \right)=1
\]
\end{proposition}

\begin{proof}
  We note that when $k_1,k_2 \le 0$, the diffusion with generator
  \[
    L_{k_1,k_2}= \big(F_{\mathrm{Sas}}(r,k_1) + (n-2)
    F_{\mathrm{Rie}}(r,k_2)\big) \frac{\partial}{\partial r}
    +\frac{\partial^2}{\partial r^2}
  \]
  is transient and that
  \[
    \lim_{r \to +\infty}F_{\mathrm{Sas}}(r,k_1)=\sqrt{|k_1|} , \quad
    \lim_{r \to +\infty}F_{\mathrm{Rie}}(r,k_2))=\sqrt{|k_2|}.
  \]
  The result follows then from similar arguments as in Example 2.1 in
  \cite{Ich}.
\end{proof}

When $k_1,k_2=0$, the above estimate can be refined and we obtain a
law of iterated logarithm.

\begin{proposition}
  Let $(\M,\mathcal{F},g)$ be a Sasakian foliation.  Assume
  that
  $$\mathbf{K}_{\mathcal{H},J}(v,v) \ge 0, \quad
  \mathbf{Ric}_{\mathcal{H},J^\perp}(v,v) \ge 0, \quad v \in \ch,\ \|
  v\|_g = 1.$$ Let $( ( \xi_t )_{t \ge 0} , ( \P_x )_{x \in \M} )$ be
  the sub-Riemannian Brownian motion generated by $\Delta_{\ch}$.
  Then for every $x_0,x_1 \in \M$,
  \[
    \P_{x_1} \left( \limsup_{t \to +\infty} \frac{d_0(x_0 ,
        \xi_t)}{\sqrt{2t \ln \ln t} } \le 1 \right)=1
  \]
\end{proposition}

\begin{proof}
  The proof is similar to the one of Theorem 3.1 in \cite{Ich}, so we
  omit the details for conciseness.
\end{proof}

\subsection{Some extensions: H-type groups}

In the recent work $\cite{BGMR}$ sub-Laplacian comparison theorems
have been obtained in a more general setting than Sasakian, the
setting of H-type sub-Riemannian spaces. In particular, this was
proved that if $(\M, \mathcal F)$ is the totally geodesic foliation on
a H-type group, then one has
\[
  \Delta_\Ho r_0 \le \frac{n+3m-1}{r_0}
\]
classically outside of the cut-locus and globally in the sense of distributions,
where as usual $n$ denotes the dimension of the horizontal bundle and
$m$ denotes the codimension of this horizontal bundle. In that
setting, all the results obtained in this section may be generalized
with identical proofs. In particular, one obtains the following
Cheng's type theorem for the Dirichlet eigenvalues of sub-Riemannian
balls in H-type groups.

\begin{theorem}
  Assume that $\M$ is an H-type group. For $x_0 \in \M$ and $R>0$, let
  $\lambda_1( B_0(x_0,R))$ denote the first Dirichlet eigenvalue of
  the sub-Riemannian ball $B_0(x_0,R)$ and let
  $\tilde{\lambda}_1(d,R)$ denote the first Dirichlet eigenvalue of
  Euclidean ball with radius $R$ in $\mathbb{R}^d$. Then, for every
  $x_0 \in \M$ and $R>0$,
  \[
    \lambda_1( B_0(x_0,R)) \le \tilde{\lambda}_1( n+3m,R).
  \]
\end{theorem}

Note that in contrast to the Riemannian case, here the comparison diffusion is not realized as the radial process of a sub-Riemannian manifold (and in particular, it is not the radial process of the Heisenberg group, which fails to be a diffusion), and thus there is no reason to suspect that the bound is sharp.

\end{document}